\documentclass{amsart}
\usepackage{amssymb, latexsym}

\DeclareMathOperator{\Tr}{Tr}
\DeclareMathOperator{\Alt}{Alt}
\DeclareMathOperator{\Bil}{Bil}

\DeclareMathOperator{\M}{\mathcal{M}}
\DeclareMathOperator{\N}{\mathcal{N}}

\DeclareMathOperator{\rank}{rank}
\DeclareMathOperator{\rad}{rad}

\DeclareMathOperator{\A}{\mathcal{A}}

\DeclareMathOperator{\Symm}{Symm}

\theoremstyle{plain}
\newtheorem{theorem}{Theorem}
\newtheorem{corollary}{Corollary}
\newtheorem{lemma}{Lemma}

\theoremstyle{definition}
\newtheorem{definition}{Definition}
\begin{document}
\title[Rank and dimension in subspaces]
{Connections between rank and dimension \\ 
for subspaces of bilinear forms}

\author[R. Gow]{Rod Gow}
\address{School of Mathematics and Statistics\\
University College Dublin\\
 Ireland}
\email{rod.gow@ucd.ie}


\begin{abstract} 
Let $K$ be a field and let $V$ be a vector space of dimension $n$ over $K$. Let $\Bil(V)$ denote the vector space of bilinear forms defined on $V\times V$.
Let $\M$ be a subspace of $\Bil(V)$ and let $\rank(\M)$ denote the set of different 
positive integers that occur as the ranks of the non-zero elements of $\M$.
Setting $r=|\rank(\M)|$, our aim is to obtain an upper bound for $\dim \M$ in terms of $r$ and $n$ under various hypotheses. As a sample of what we prove, we mention the following. Suppose that $m$ is the largest integer in $\rank(\M)$. Then if $m\leq \lceil n/2\rceil$ and $|K|\geq m+1$, we have 
$\dim \M\leq rn$. The case $r=1$ corresponds to a constant rank space and it is conjectured that $\dim \M\leq n$ when $\M$ is a constant rank $m$ space and $|K|\geq m+1$.
We prove that the dimension bound for a constant rank $m$ space $\M$ holds provided $|K|\geq m+1$ and either $K$ is finite or $K$ has characteristic different from 2 and $\M$ consists of symmetric forms. In general, we show that if $\M$ is a constant rank $m$ subspace and
$|K|\geq m+1$, then $\dim \M\leq \max\,(n,2m-1)$.
We also provide more detailed results about constant rank subspaces over finite fields, especially subspaces of alternating or symmetric bilinear forms.
\end{abstract}
\maketitle

\section{Introduction}

\noindent Let $K$ be a field and let $V$ be a vector space of finite dimension $n$ over $K$. We let $V^\times$ denote the subset of non-zero elements of $V$, and use similar notation for the subset of non-zero elements in any vector space.
Let 
$\Bil(V)$ 	
denote the 
$K$-vector space of all bilinear forms defined on $V\times V$.
Let $\Alt(V)$ denote the subspace of $\Bil(V)$ consisting of alternating bilinear forms and $\Symm(V)$  the subspace of symmetric bilinear forms.

Let $f$ be an element of $\Bil(V)$. Let $\rad_L f$ denote the left radical of $f$
and $\rad_R f$ denote the right radical of $f$. It is known that $\dim \rad_L f=\dim \rad_R f$ and the number $n-\dim \rad_L f$ is called the rank of $f$, which we denote by
$\rank f$. 

In the case that $f$ is alternating or symmetric,  $\rad_L f=\rad_R f$ and we call 
the common subspace the radical of $f$, denoted by $\rad f$. 

\begin{definition} \label{rank_definition}
Let $\M$ be a non-zero subspace of $\Bil(V)$. We let $\rank(\M)$ denote the set of different integers that occur as the ranks of the non-zero elements of $\M$. If $\M$
is the zero subspace, we set $\rank(\M)=0$.
\end{definition}

Thus $\rank(\M)=\{ \rank f: f\in \M^\times\}$, where we only include the different ranks that occur. Clearly, we have $\rank(\N)\leq \rank(\M)$ when $\N$ is a subspace of $\M$.

One purpose of this paper is to obtain an upper bound for $\dim \M$ in terms of $|\rank(\M)|$ and $n$. The results we obtain vary according to the nature of the forms and certain hypotheses we make, and may be subdivided into three types, described as follows.

Let $\M$ be a non-zero 
subspace of $\Bil(V)$, let $m$ be the largest integer in $\rank(\M)$ and let
$r=|\rank(\M)|$. Then if
$m\leq \lceil n/2\rceil$
and $|K|\geq m+1$, we have
\[
\dim \M\leq rn.
\]
This bound is optimal in non-trivial cases, but we do not know if the restriction
on the size of $m$ is essential. See Theorem \ref{bilinear_several_ranks_bound}. 

Suppose next that $\M$, as above,
is a subspace of $\Alt(V)$. Then if
$m\leq \lfloor n/2\rfloor$ 
and $|K|\geq m+1$, we have
\[
\dim \M\leq rn-\frac{r(r+1)}{2}.
\]
Examples show that this bound is also optimal in non-trivial ways, but the restriction on the size of $m$ is essential. See Theorem \ref{alternating_several_ranks_bound}. 

Finally, suppose that $\M$ is a subspace of $\Symm(V)$. Then if $K$ has characteristic different from 2 and $|K|\geq n$, we have
\[
 \dim \mathcal{M}\leq rn-\frac{r(r-1)}{2}.
\]
The hypothesis on the characteristic of $K$ is essential. This bound is optimal in some cases, but not so in general. For small values of $r$ and specific fields, examples show that the bound is reasonably precise. See Theorem \ref{symmetric_general_dimension_bound}.

Our original motivation for undertaking investigations of this nature is as follows.
Suppose that $|\rank(\M)|=1$, and let $m$ be the rank of all non-zero elements of $\M$. We say  that $\M$ is a constant rank $m$ subspace of $\Bil(V)$.

It is conjectured that the dimension of a constant rank $m$ subspace is at most $n$, provided that $|K|\geq m+1$. For many fields, including all finite fields, there are constant rank $m$ subspaces of dimension $n$, when $1\leq m\leq n$, and thus the conjectured upper bound is optimal if it is proved to hold.
This dimension bound is trivial to prove for all fields if $m=n$
and thus interest is concentrated on the case when
$m<n$. 

The dimension bound is a consequence of Theorem 7 for a constant rank $m$ subspace of $\Symm(V)$ when $K$ has characteristic different from 2 and $|K|\geq m+1$. We also prove the upper bound for finite fields
of size at least $m+1$ (Theorem \ref{finite_constant_rank}). We  already proved an equivalent theorem in \cite{G1}. 

We supplement this dimension bound with the following additional information when the maximum dimension $n$ occurs. Let $\M$ be an $n$-dimensional constant rank $m$ subspace of $\Bil(V)$ and let $K=\mathbb{F}_q$, where  $q\geq m+1$. Then if $n\geq 2m+1$, all elements of $\M^\times$ either have the same left radical or the same right radical.
See Theorem \ref{equality_of_left_radicals}. Some condition on the size of $m$ relative to $n$ is necessary for the truth of this theorem.

Concerning the general case of the constant rank dimension bound, we prove the following results. Let $\M$ be a constant rank $m$ subspace of $\Bil(V)$.  Then if $|K|\geq m+1$, we have $\dim \M\leq \max\,(n,2m-1)$. See Theorem \ref{bilinear_constant_rank_bound}. If we assume that $\M\leq \Alt(V)$, this bound can be improved to $\dim \M\leq \max\,(n-1,2m-1)$.
See Theorem \ref{alternating_constant_rank_bound}.

We have remarked above that if $\M$ is a constant rank $m$ subspace of $\Bil(V)$, we have
$\dim \M\leq n$ if we work over a sufficiently large finite field, and this upper bound is optimal. For constant rank subspaces of $\Alt(V)$ and $\Symm(V)$
in the finite field case, there is reason to suppose that the upper bound  of $n$ for the dimension can often be improved. 

To illustrate this point, let $\M$ be a constant rank $m$ subspace of $\Alt(V)$
and $K=\mathbb{F}_q$,  where $q\geq m+1$. Then if
$4\leq m\leq
\lfloor n/2\rfloor$, we have $\dim \M\leq n-2$
(Theorem \ref{new_alternating_dimension_bound}).
Similarly, 
let $\M$  be a constant rank $m$ subspace of $\Symm(V)$ and $K=\mathbb{F}_q$, where
$q$ is odd and at least $m+1$. Then  if $m\leq 
2n/3$, we have $\dim \M\leq n-1$ (Theorem \ref{the_main_dimension_bound}). 

Additional results can  also be obtained for certain constant rank subspaces of
$\Alt(V)$ in the finite field case. By way of example, let $\M$ be an $n$-dimensional constant rank $m$ subspace of
$\Alt(V)$ and let  $K=\mathbb{F}_q$. Then if $q\geq m+1$, $n-m$ divides $n$. This follows from the fact that the different subspaces of
dimension $n-m$ of $V$ that occur as the radicals of the elements of $\M^\times$ form a spread of $V$ (Theorem \ref{spread_theorem}). We also show that such $n$-dimensional
constant rank $m$ subspaces of $\Alt(V)$ exist if $n$ is odd and $n-m$ divides $n$.

\section{Basic theorems for studying bilinear forms}

\noindent The following is fundamental to all the results we obtain in this paper.

\begin{theorem} \label{left_radical_right_radical}

Let $\M$ be a subspace of $\Bil(V)$ and let $m$ be the largest integer in $\rank(\M)$. Let $f$ be an element of $\M$ with $\rank f=m$. Let $u$, $w$ be arbitrary elements
of $\rad_L f$, $\rad_R f$, respectively. Then if $|K|\geq m+1$, we have 
\[
g(u,w)=0
\]
for all elements $g$ of $\M$. 

\end{theorem}

\begin{proof}
We set 
$U=\rad_L f$ and $W=\rad_R f$. These are both subspaces of dimension $n-m$.
It follows that there is an automorphism, $\sigma$, say, of $V$ with $\sigma(U)=W$.
For each element $g$ of $\M$, we define $g_\sigma$ in $\Bil(V)$ by setting
\[
g_\sigma(x,y)=g(x, \sigma y)
\]
for all $x$ and $y$ in $V$. Clearly, since $\sigma$ is an automorphism of $V$, $g_\sigma$ has the same rank as $g$. 

Thus $f_\sigma$ has rank $m$ and its left radical is $U$. The right radical of $f_\sigma$ consists of those elements $y$ such that $\sigma y\in W$. Thus $y\in \sigma^{-1}(W)=U$. We see therefore that $U$ is both the left and right radical of $f_\sigma$.

Let $U'$ be a complement for $U$ in $V$. With respect to a basis of $V$ consisting of bases of $U$ and $U'$, we can take the matrix of $f_\sigma$ to be 
\[
C=\left(
\begin{array}
{cc}
        0&0\\
        0&A   
\end{array}
\right),
\]
where $A$ is an invertible $m\times m$ matrix.

Given $g$ in $\M$, let the matrix of $g_\sigma$ with respect to the same basis be 
\[
D=\left(
\begin{array}
{cc}
        A_1&A_2\\
        A_3&A_4   
\end{array}
\right),
\]
where $A_1$ is an $(n-m)\times (n-m)$  matrix,  $A_4$ is an $m\times m$ matrix, 
and $A_2$, $A_3$ are matrices of the appropriate compatible sizes. 

When we follow the proof of Theorem 1 of \cite{G2}, we find that $A_1=0$ if 
$|K|\geq m+1$. Now the fact that $A_1=0$ means that 
\[
g_\sigma(x,y)=0
\]
for all $x$ and $y$ in $U$. It follows that
\[
g(x,\sigma y)=0
\]
for all $x$ and $y$ in $U$. However, since $\sigma(U)=W$, we obtain
\[
g(u,w)=0
\]
for all $u$ in $U$ and all $w$ in $W$.
\end{proof}

Next, we introduce a simple idea which is useful for induction arguments  that establish a relationship between $\dim \M$ and $|\rank(\M)|$. 

Let $V^*$ denote the dual space of $V$ and let $u$ be any vector in $V$.
We define a linear transformation $\epsilon_u:\M\to V^*$ by setting
\[
\epsilon_u(f)(v)=f(u,v)
\]
for all $f\in \M$ and all $v\in V$. Likewise, we define a linear transformation $\eta_u:\M\to V^*$ by setting
\[
\eta_u(f)(v)=f(v,u)
\]
for all $f\in \M$ and all $v\in V$.

Let $\M_u^L$ denote the kernel of $\epsilon_u$. This is the subspace of $\M$ consisting of all those forms $f$ such that $u\in \rad_L f$. Similarly, 
let $\M_u^R$ denote the kernel of $\eta_u$. This is the subspace of $\M$ consisting of all those forms $f$ such that $u\in \rad_R f$.

When $\M$ is a subspace of either $\Alt(V)$ or $\Symm(V)$, clearly $\M_u^L=\M_u^R$. In this case, we write $\M_u$ in place of $\M_u^L$. 

The following estimate for $\dim \M_u^L$ and $\dim \M_u^R$ follows from the fact that $\dim V^*=n$. See, for example, Lemma 2 of \cite{G4}.

\begin{lemma} \label{M_u_dimension_bound}

For each element $u$ of $V$, we have
\[
\dim \M_u^L\geq \dim \M-n, \quad \dim \M_u^R\geq \dim \M-n.
\]
\end{lemma}

The following more precise estimate for both $\dim \M_u^L$ and $\dim \M_u^R$ is also important throughout this paper. Its proof relies on Theorem \ref{left_radical_right_radical}.

\begin{lemma} \label{improved_M_u_dimension_bound}
Let $\M$ be a non-zero subspace of $\Bil(V)$ and let $m$ be the largest integer in $\rank(\M)$. Let $u$ be an element of $V$.
Suppose that $\M_u^L$ contains an element of rank $m$. Then if $|K|\geq m+1$, we have 
\[
\dim \M_u^L\geq \dim \M-m.
\]
Furthermore, if $\dim \M_u^L=\dim \M-m$, then all elements of rank $m$ in $\M_u^L$ have the same right radical.
Similarly, if  $\M_u^R$ contains an element of rank $m$ and if $|K|\geq m+1$,
\[
\dim \M_u^R\geq \dim \M-m.
\]
The equality $\dim \M_u^R=\dim \M-m$ implies that all elements of rank $m$ in
$\M_u^R$ have the same left radical.
\end{lemma}

\begin{proof}
Suppose that $\M_u^L$ contains an element $f$, say, of rank $m$. Then provided that
$|K|\geq m+1$, Theorem \ref{left_radical_right_radical} implies that as $u\in \rad_L
f$, 
\[
g(u,w)=0
\]
for all $w$ in $\rad_R f$ and all $g$ in $\M$. We deduce that $\epsilon_u(\M)$ is contained in the annihilator of $\rad_R f$ in $V^*$. Now since $\rad_R f$ has dimension $n-m$, its annihilator in $V^*$ has dimension $m$. This establishes that $\dim \epsilon_u(\M)\leq m$. It is also clear that if $\dim \epsilon_u(\M)=m$, $\epsilon_u(\M)$ is  precisely the annihilator 
of $\rad_R g$ for each element $g$ of rank $m$ in $\M_u^L$, and hence these right radicals are all the same.

When we recall that 
\[
\dim \M=\dim  \epsilon_u(\M)+\dim \M_u^L,
\]
we obtain the desired inequality $\dim \M_u^L\geq \dim \M-m$.

An identical proof serves to estimate $\dim \M_u^R$ under the stated hypotheses and to identify left radicals when equality holds.
\end{proof}

\section{Constant rank subspaces in finite field case}

\noindent The following lemma is our main tool to investigate constant rank subspaces of bilinear forms over finite fields.

\begin{lemma} \label{constant_rank_dimension_lemma}

Let $\M$ be a $d$-dimensional constant rank $m$ subspace of $\Bil(V)$ and let
$K=\mathbb{F}_q$. Then we have
\[
(q^d-1)(q^{n-m}-1)=\sum_{u\neq 0} (q^{d(u)}-1),
\]
where the sum extends over all non-zero vectors $u$  and $d(u)=\dim \M_u^L$.
\end{lemma}

\begin{proof}
Let $\Omega$ be the set of pairs $(f,u)$, where $f$ is an element of $\M^\times$, and $u$ is an element of $(\rad_L f)^\times$.  
We first evaluate $|\Omega|$ by fixing $f$ and counting those $u\neq 0$ in $\rad_L f$. We obtain
\[
|\Omega|=(q^d-1)(q^{n-m}-1).
\]

Next we evaluate $|\Omega|$ by fixing a non-zero $u$ and counting those $f$ with
$u\in \rad_L f$. The required  $f$  are the non-zero elements of $\M_u^L$. Thus, summing over all non-zero $u$, we derive the equality
\[
|\Omega|=\sum_{u\neq 0} (q^{d(u)}-1).
\]
This proves what we want.  
\end{proof}

This lemma enables us to prove that a constant rank subspace of $\Bil(V)$ has dimension at most $n$, provided that we work over a sufficiently large finite field.

\begin{theorem} \label{finite_constant_rank}

Let $\M$ be a constant rank $m$ subspace of $\Bil(V)$ and let
$K=\mathbb{F}_q$, where $q\geq m+1$. Then we have $\dim \M\leq n$. (When $m=n$, the result is true for all fields $K$.)
\end{theorem}

\begin{proof}

Suppose  if possible that $\dim \M>n$. Then we may as well assume that $\dim \M=n+1$ and proceed to derive a contradiction. Given $u\in V$, we set $d(u)=\dim \M_u^L$. Since we are assuming that $\dim \M>n$, Lemma \ref{M_u_dimension_bound} implies that
$d(u)\geq 1$. We deduce from Lemma \ref{improved_M_u_dimension_bound} that, since $\M$ is a constant rank $m$ subspace, 
\[
d(u)\geq n+1-m.
\]

Lemma \ref{constant_rank_dimension_lemma} shows that 
\[
\sum_{u\neq 0} (q^{d(u)}-1)=(q^{n+1}-1)(q^{n-m}-1).
\]  
On expanding each side above, we obtain
\[
(\sum_{u\neq 0} q^{d(u)})-q^n=q^{2n+1-m}-q^{n+1}-q^{n-m}.
\]
The highest power of $q$ dividing the right hand side is $q^{n-m}$. On the other hand, as $d(u)\geq n+1-m$, the highest power of $q$ dividing the left hand side is
at least $q^{n+1-m}$. This is a contradiction, and we deduce that $\dim \M\leq n$.
\end{proof}

We gave a similar proof of an equivalent theorem in \cite{G1} but have included this proof to illustrate  the ideas of bilinear form theory.

Our next objective is to investigate what happens when we have the equality $\dim \M=n$ in Theorem \ref{finite_constant_rank}. We begin by defining two relevant subspaces of $V$. 

\begin{definition} \label{V^L_M_definition}

Let $\M$ be a subspace of $\Bil(V)$. We set
\[
V(\M)^L=\{ v \in V: \M_v^L\neq 0\} \mbox{ and } V(\M)^R=\{v \in V: \M_v^R\neq 0\}.
\]

\end{definition}

\begin{lemma} \label{V(M)_is_a_subspace}

Let $\M$ be a constant rank $m$ subspace of $\Bil(V)$,  with $\dim \M\geq 2m+1$. Then, if $|K|\geq m+1$,  $V(\M)^L$ and $V(\M)^R$ are both subspaces of $V$.

\end{lemma}

\begin{proof}
Let $u$, $w$ be  elements of $V(\M)^L$.  We wish to show that $\M_{u+w}^L\neq 0$.
Then, since $V(\M)^L$ is clearly closed
 under scalar multiplication, it will follow that  $V(\M)^L$ is a subspace.
 
 We claim that if we can show that $\M_u^L\cap \M_w^L\neq 0$,  this will establish that
 $\M_{u+w}^L\neq 0$. For suppose that $g$ is a non-zero element of $\M_u^L\cap \M_w^L$. Then  $u$ and $w$ are contained in $\rad_L g$ and hence $u+w\in \rad_L g$. This implies that $\M_{u+w}^L\neq 0$, as required.

We turn therefore to proving that $\M_u^L\cap \M_w^L\neq 0$.
Lemma \ref{improved_M_u_dimension_bound} shows that
both $\M_u^L$ and $\M_w^L$ have dimension at least $\dim \M-m$. In addition, we have the inequality
\[
\dim(\M_u^L+\M_w^L)\leq \dim \M
\]
and hence
\[
\dim \M_u^L+\dim \M_w^L
-\dim (\M_u^L\cap \M_w^L)\leq \dim \M.
\]

Given the earlier inequalities for $\dim \M_u^L$ and $\dim \M_w^L$, we deduce that
\[
2(\dim \M-m)-\dim (\M_u^L\cap \M_w^L) \leq \dim \M
\]
and hence
\[
\dim (\M_u^L\cap \M_w^L)\geq \dim \M-2m.
\]
Since we are assuming that  $\dim \M\geq 2m+1$, we see that $\dim (\M_u^L\cap \M_w^L)\geq 1$, and this completes the proof that $V(\M)^L$ is a subspace. The proof that
$V(\M)^R$ is also a subspace is identical, since the same inequalities hold.
\end{proof}

\begin{lemma} \label{left_orthogonal_right_orthogonal}

Let $\M$ be a constant rank $m$ subspace of $\Bil(V)$, with $\dim \M\geq 2m+1$. Then,  if $|K|\geq m+1$, we have 
\[
f(u,w)=0
\]
for all  $u\in V(\M)^L$, all $w\in V(\M)^R$, and all $f\in \M$.

\end{lemma}

\begin{proof}

Let $u$ and $w$ be elements in $V(\M)^L$, $V(\M)^R$, respectively. Then we have
$\dim \M_u^L\geq \dim \M-m$, $\dim \M_w^R\geq \dim \M-m$ by Lemma \ref{improved_M_u_dimension_bound}. The dimension argument used in Lemma
\ref{V(M)_is_a_subspace} shows that $\M_u^L\cap \M_w^R\neq 0$, since we are assuming that $\dim \M\geq 2m+1$. 

Let $g$ be a non-zero element in $\M_u^L\cap \M_w^R$. Then $u$ is in $\rad_L g$,
$w$ is in $\rad_R g$ and hence Theorem \ref{left_radical_right_radical} implies that
\[
f(u,w)=0
\]
for all $f$ in $\M$. This proves the lemma.
\end{proof}

\begin{lemma} \label{dimension_estimate_in_terms_of_V}

Let $\M$ be a constant rank $m$ subspace of $\Bil(V)$,  with $\dim \M\geq 2m+1$. Then, if $|K|\geq m+1$, we have 
\[
\dim \M_u^L\geq \dim \M-n+\dim V(\M)^R
\]
for all  $u\in V(\M)^L$.

\end{lemma}

\begin{proof}
Lemma \ref{left_orthogonal_right_orthogonal} implies that $\epsilon_u(\M)$ annihilates $V(\M)^R$. Thus $\epsilon_u(\M)$ is contained in the annihilator of 
$V(\M)^R$ in $V^*$. We deduce that
\[
\dim \epsilon_u(\M)\leq n-\dim V(\M)^R.
\]
Since $\dim \epsilon_u(\M)=\dim \M-\dim \M_u$, the inequality follows.
\end{proof}

It is clear that if $\M$ is a subspace of $\Bil(V)$, $V(\M)^L$ is the union of the left radicals of the elements of $\M$, and similarly $V(\M)^R$ is the union of the right radicals. Thus, if $\M$ is a constant rank $m$ subspace of $\Bil(V)$ and if $V(\M)^L$ is a subspace of $V$, certainly $\dim V(\M)^L\geq n-m$, and $\dim V(\M)^L= n-m$ if and only if all elements of $\M^\times$ have the same left radical. Similarly, if $V(\M)^R$ is a subspace, its dimension is at least $n-m$ and it equals $n-m$ if and only if all elements of $\M^\times$ have the same right radical. 

We proceed now to prove a theorem about the equality of left
radicals or of right radicals for constant rank $m$ subspaces of maximum dimension $n$ provided we assume that $n\geq 2m+1$ and we work over sufficiently large finite
fields.

\begin{theorem} \label{equality_of_left_radicals}

Let $\M$ be an $n$-dimensional constant rank $m$ subspace of $\Bil(V)$ 
and let $K=\mathbb{F}_q$, where $q\geq m+1$. Then
if $n\geq 2m+1$, the elements of $\M^\times$ either all have the same left radical or they have the same right radical.

\end{theorem}

\begin{proof}
Let us assume that our assertion above about the left and right radicals is not true.
Then we have the inequalities $\dim V(\M)^L\geq n-m+1$, $\dim V(\M)^R\geq n-m+1$ by the discussion after the proof of Lemma \ref{dimension_estimate_in_terms_of_V}, and we  will show that these lead to a contradiction. 

Lemma \ref{constant_rank_dimension_lemma} shows that
\[
\sum_{u\neq 0} (q^{d(u)}-1)=(q^{n}-1)(q^{n-m}-1).
\]  
We are interested only
in those $u$ for which $d(u)>0$ and these are the elements of $V(\M)^L$. 
Let us put $\dim V(\M)^L=d$. Then
\[ 
\sum_{u\neq 0} (q^{d(u)}-1)=(\sum_{u\neq 0} q^{d(u)})-q^d+1.
\]
On expanding and rearranging, we obtain
\[
(\sum_{u\neq 0}q^{d(u)})-q^d=q^{2n-m}-q^{n}-q^{n-m}.
\]
We are assuming that $d\geq n-m+1$ and  also that $\dim V(\M)^R\geq n-m+1$. We 
then have $d(u)\geq n-m+1$ by
Lemma \ref{dimension_estimate_in_terms_of_V}.
Thus the power of $q$ dividing the left hand side above is at least $q^{n-m+1}$. However, the power of $q$ dividing the right hand side is exactly $q^{n-m}$. We have reached a contradiction, and thus either the left radicals are all equal, or the right radicals are all equal.
\end{proof}

We note that we cannot weaken the hypothesis that $\dim \M=n$ in Theorem \ref{equality_of_left_radicals}, since, for example, there exists a constant rank 2 subspace of $\Symm(V)$ of dimension $n-1$ in which all linearly independent elements have different radicals. Similarly, some restriction on the size of $m$ compared with $n$ is needed, since, as we shall see in Section 7, there exist $n$-dimensional constant rank $m$ subspaces of $\Alt(V)$, for which Theorem \ref{equality_of_left_radicals} cannot possibly be true. The simplest example occurs when $n=3$. $\Alt(V)$ is a three-dimensional constant rank 2 subspace, in which linearly independent forms have different one-dimensional radicals.

\section{Dimension bounds for subspaces of symmetric forms}

\noindent   We begin this section on symmetric  forms with an application of Theorem \ref{left_radical_right_radical}.

\begin{lemma} \label{symmetric_no_elements_of_maximum_rank}

Let $K$ be a field of characteristic different from $2$ and let $\M$ be a non-zero subspace of $\Symm(V)$.  
Let $m$ be the largest integer in
$\rank(\M)$. Then if $|K|\geq m+1$, there exists an element $u$, say, in $V$ such that $\M_u$ contains no element of rank $m$.

\end{lemma}

\begin{proof}
There exists an element $h\in \M$ and $u\in V$ with $h(u,u)\neq 0$, 
since $K$ has characteristic different from 2 and $\M\neq 0$ consists of symmetric bilinear forms.
Suppose now that $\M_u$ contains an element $f$, say, of rank $m$. Then $u\in \rad f$ and hence Theorem \ref{left_radical_right_radical} implies that, if we assume that
$|K|\geq m+1$, each element $g$ of $\M$ satisfies
\[
g(u,u)=0.
\]
This contradicts our earlier statement about the existence of $h$ in $\M$ with $h(u,u)\neq 0$.
We deduce that  $\M_u$ contains no element of rank $m$, as required.
\end{proof}

We have enough information to prove our constant rank dimension bound for subspaces of 
$\Symm(V)$.

\begin{theorem} \label{symmetric_constant_rank_bound}

Let $\M$ be a constant rank $m$ subspace of $\Symm(V)$. Then if $K$ has characteristic different from $2$ and $|K|\geq m+1$, we have $\dim \M\leq n$. (When $m=n$, the bound $\dim \M\leq n$ holds for all fields $K$ without exception, as noted in Theorem \ref{finite_constant_rank}.)

\end{theorem}

\begin{proof}
 We suppose that $|K|\geq m+1$. 
 Clearly, by the constant rank hypothesis, Lemma \ref{symmetric_no_elements_of_maximum_rank} implies that there is some $u$ in $V$
 such that $\M_u=0$. Lemma \ref{M_u_dimension_bound} implies then that $\dim \M\leq n$, as required.
\end{proof}

We would like now to show that, in contrast to Theorem \ref{symmetric_constant_rank_bound}, given positive integers
$m$ and $n$, with $2\leq m\leq n$, and making certain assumptions about $K$, $\Symm(V)$ contains a constant rank $m$ subspace $\M$ of dimension $m$ that is maximal with respect to the property of being constant rank $m$. In other words, $\M$ is not contained in any larger constant rank $m$ subspace
of $\Symm(V)$. Thus for example, under fairly weak hypotheses on $K$, there are 
two-dimensional maximal constant rank 2 subspaces of $\Symm(V)$. The construction is based on simple concepts of field theory.

Suppose that $K$ has a separable extension $L$, say, of degree $m\geq 2$, but is otherwise arbitrary. We consider $L$ as a vector space of dimension $m$ over $K$. Let
$\Tr:L\to K$ denote the trace form. 

For each element $z$ of $L$, we define an element $f_z$ of $\Symm(L)$ by setting 
\[
f_z(x,y)=\Tr(z(xy))
\]
for all $x$ and $y$ in $L$. 

Let $\N$ be the subspace of $\Symm(L)$ consisting of all the $f_z$. We have $\dim \N=m$, and each element of $\N^\times$ has rank $m$, since $\Tr$ is non-zero under the hypothesis of separability. 

We note the following property of $\N$ whose (omitted) proof depends on the fact that $\Tr$ is non-zero.

\begin{lemma} \label{non_vanishing_property}

Let $\N$ be the subspace of $\Symm(L)$ described above. Let $x$ and $y$
be elements of $L$ such that 
\[
f_z(x,y)=0
\]
for all $f_z$ in $\N$. Then $x=0$ or $y=0$.
\end{lemma}

We can now construct our example.

\begin{theorem} \label{maximal_example}

Suppose that $K$ has a separable extension of degree $m\geq 2$. Then if $m\leq n=\dim V$
and $|K|\geq m+1$, 
$\Symm(V)$ contains a constant rank $m$ subspace $\M$ of dimension $m$ that is contained in no larger constant rank $m$ subspace of $\Symm(V)$. Thus, $\M$ is maximal with respect to containment in constant rank $m$ subspaces.

\end{theorem}

\begin{proof}
 Let $U$ be a subspace of $V$ with $\dim U=m$, and let $W$ be a complement for $U$ in $V$. 
We may identify $U$ with $L$ as a vector space of dimension $m$ over $K$, and hence may define a constant rank $m$ subspace of $\Symm(U)$ of dimension $m$. We may then extend this subspace to a subspace $\M$
of $\Symm(V)$ by setting the extensions to have radical $W$ in all non-zero cases.

We claim that the subspace $\M$ thus constructed is maximal in $\Symm(V)$ with respect to being constant rank $m$. For suppose that $\M$ is contained in a larger constant rank $m$ subspace, $\M_1$, say, of $\Symm(V)$. Let $f$ be an element of $\M_1$ not contained in $\M$. We aim to show that $\rad f=W$.

Now since $|K|\geq m+1$, Theorem \ref{left_radical_right_radical} implies that $\rad f$ is totally isotropic for $\M$. Let $v$ be any element of $\rad f$ and set $v=u+w$, where $u\in U$ and $w\in W$. Let $g$ be any element of $\M$. Then we have $g(v,v)=0$
and hence
\[
g(u+w,u+w)=g(u,u)=0,
\]
since $w\in \rad g$. Thus $g(u,u)=0$ for all $g\in \M$. Lemma \ref{non_vanishing_property} implies that $u=0$.

We see therefore that $\rad f\leq W$ and hence $\rad f=W$, by consideration of dimensions. We have thus proved that all  elements of $\M_1^\times$ have radical
$W$.

It is permissible then to identify $\M_1$ with a constant rank $m$ subspace
of $\Symm(V/W)$. Since $V/W$ has dimension $m$, it follows that $\dim \M_1\leq m$.
This is a contradiction, and our claim that $\M$ is maximal is established.
\end{proof}

When we work over the field of real numbers, we can show that there are one-dimensional constant rank $m$ subspaces that are maximal for any positive integer value of $m$. Our next theorem provides the details.

\begin{theorem} \label{real_example}
Let $K$ be the field of real numbers
and let $m$ be a positive integer with $m\leq n$. Let $\M$ be a constant rank $m$ subspace of $\Symm(V)$ that contains a non-zero positive semidefinite element. Then $\dim \M=1$.
\end{theorem}

\begin{proof}
Let $f\neq 0$ be a positive semidefinite element in $\M$ and let $U$ be a complement to $\rad f$ in $V$. Then, $f$ is positive definite on $U\times U$. Now let $g$ be any non-zero element
in $\M$ and let $R=\rad g$. Theorem \ref{left_radical_right_radical} implies that
$f$ is zero on $R\times R$. Hence $f$ is also zero on $(R+\rad f)\times (R+\rad f)$.

Since $f$ is positive definite on $U\times U$, and zero on $(R+\rad f)\times (R+\rad f)$,
we have $U\cap (R+\rad f)=0$. It follows that $R=\rad f$ and thus all non-zero elements
of $\M$ have the same radical, which we may take to be 0. We may thus assume that all non-zero elements have maximum rank $n$.

Suppose if possible that $\dim \M\geq 2$. Let $g$ be an element of $\M$ linearly independent of $f$. By a well known theorem of linear algebra over the real numbers,
there is a basis of $V$ that is orthonormal with respect to $f$ and orthogonal
with respect to $g$. But then it follows easily that there is a linear combination of $f$ and $g$ that has non-zero radical, contradicting our statement in the paragraph above.
We deduce that $\dim \M=1$, as required.
\end{proof}

\begin{corollary} \label{one_dimensional_example}
Let $K$ be the field of real numbers
and let $m$ be a positive integer with $m\leq n$. Then there exists a one-dimensional maximal constant rank $m$ subspace of $\Symm(V)$.
\end{corollary}

We return now to the theme of bounding $\dim \M$ in terms of $|\rank(\M)|$. We investigate
subspaces of $\Symm(V)$ and employ the idea underlying the proof of Theorem \ref{symmetric_constant_rank_bound}. 

\begin{theorem} \label{symmetric_general_dimension_bound}
Let $\mathcal{M}$ be a subspace of $\Symm(V)$ and let $r=|\rank(\M)|$. Then if $K$ has characteristic different from $2$ and $|K|\geq n$, we have
\[
 \dim \mathcal{M}\leq rn-\frac{r(r-1)}{2}.
\]
\end{theorem}

\begin{proof}
We proceed by induction on $r$. The result is trivially
true when $r=0$, so we can therefore assume that $r\geq 1$. 

Let $m$ be the largest integer in $\rank(\M)$. Lemma \ref{symmetric_no_elements_of_maximum_rank} implies that there exists some element
$u$ in $V$ such that $\M_u$ contains no element of rank $m$. Thus, $|\rank(\M_u)|\leq r-1$.

Let $U$ be the one-dimensional subspace of $V$ spanned by $u$ and let $U'$ be a complement of $U$ in $V$. Since $U$ is in the radical of each element 
of $\M_u$, we may identify $\M_u$ with a subspace of $\Symm(U')$. 

Since $\dim U'=n-1$, and  $|\rank(\M_u)|\leq r-1$, we have by induction that
\[
\dim \M_u\leq (r-1)(n-1)-\frac{(r-1)(r-2)}{2}.
\]

Lemma \ref{M_u_dimension_bound} yields that $\dim \M\leq \dim \M_u+n$ and we obtain the bound
\[
\dim \M\leq (r-1)(n-1)-\frac{(r-1)(r-2)}{2}+n=rn-\frac{r(r-1)}{2},
\]
as required.
\end{proof}

\bigskip

\noindent{\bf Example 1.}
Let $r$ be a positive integer such that $r\leq n/2$. Let $\M'$ be the subspace of all $n\times n$ symmetric matrices of the form 
\[
\left(
\begin{array}
{cc}
        0&A\\
        A^T&0   
\end{array}
\right),
\]
where $A$ runs over all $r\times (n-r)$ matrices with entries in $K$, and $A^T$ denotes the transpose of $A$. 

Such a matrix has rank $2s$, where $s$ is the rank of $A$ (and hence $0\leq s\leq r$). Thus $|\rank(\M')|=r$ and $\dim \M'=r(n-r)$. $\M'$ defines a subspace $\M$, say, of $\Symm(V)$, with $\dim \M=\dim \M'=r(n-r)$ and $|\rank(\M)|=|\rank(\M')|=r$.

The bound for $\dim \M$ given by our theorem differs from the exact dimension by $r(r+1)/2$. Thus, for small values of $r$, we have a reasonably accurate dimension bound.

\bigskip
\noindent {\bf Example 2.}
It is straightforward to show that a real symmetric matrix of trace 0 cannot have rank one. It follows that if $V$ is a vector space of dimension $n$ over the field
of real numbers, $\Symm(V)$ contains a subspace $\M$ of dimension $n(n+1)/2-1$ in which $|\rank(\M)|=n-1$. This shows that Theorem \ref{symmetric_general_dimension_bound}
is precise in the case $|\rank(\M)|=n-1$ for subfields of the real numbers.

\section{Constant rank subspaces of symmetric forms over finite fields}

\noindent Theorem \ref{symmetric_constant_rank_bound} shows that the dimension
of a constant rank subspace of $\Symm(V)$ is at most $n$, provided that the underlying field is sufficiently large and has characteristic different from 2. A lack of specific examples or of construction processes suggests that this upper bound is rarely obtained. This section is devoted to improving Theorem \ref{symmetric_constant_rank_bound}, although our definitive results are currently restricted to finite fields, as ultimately we employ counting techniques to complete our arguments.

We begin with definitions of concepts which were implicit in the previous section.

\begin{definition} \label{totally_isotropic_subspace}

Let $\M$ be a subspace of $\Symm(V)$. We say that a subspace $U$ of $V$ is totally isotropic for $\M$ if
\[
f(u,w)=0
\]
for all $u$ and $w$ in $U$, and all $f$ in $\M$.

\end{definition}

\begin{definition} \label{isotropic_points}

Let $\M$ be a subspace of $\Symm(V)$. We set 
\[
I(\M)= \{ w\in V: f(w,w)=0 \mbox{ for all } f\in \M\}.
\]
We also set $I(\M)^\times$ to be the subset of non-zero elements in $I(\M)$.

\end{definition}

Clearly, a non-zero vector is in $I(\M)^\times$ if and only if the one-dimensional  subspace it spans is totally isotropic for $\M$. More generally, a subspace of $V$ that is totally isotropic for $\M$ is contained in $I(\M)$. 

We consider the converse next. 

\begin{lemma} \label{totally_isotropic_in_subspace}
Let $U$ be a subspace 
of $V$ that is contained in $I(\M)$. Then $U$ is totally isotropic for $\M$
provided that $K$ has characteristic different from $2$.

\end{lemma}

\begin{proof}
Let $u$ and $w$ be elements of $U$ and $f$ be an element of $\M$. Then since $U$ is contained in $I(\M)$, we have
\[
f(u,u)=f(w,w)=f(u+w,u+w)=0
\]
and hence $2f(u,w)=0$, using the symmetry of $f$. We thus have $f(u,w)=0$ when $K$ has characteristic different from 2, and $U$ is totally isotropic for $\M$.
\end{proof}

We remark that in this section, the key point of one argument involves the case that $I(\M)$ is itself a subspace of $V$. It is in this case that we have resorted to counting techniques to resolve our problems.

\begin{definition} \label{A_u_definition}
Let $\M$ be a subspace of $\Symm(V)$ and let $u$ be an element of $V$. We let $\A_u$
denote the subspace of $V$ annihilated by the subspace $\epsilon_u(\M)$ of $V^*$.
Thus
\[
\A_u=\{ w\in V: f(u,w)=0 \mbox{ for all }f\in \M\}.
\]
\end{definition}

The lemma that follows is important for our analysis of constant rank subspaces of $\Symm(V)$ of dimension $n$ (assuming they exist).

\begin{lemma} \label{key_lemma_isotropic}

Let $\M$ be an $n$-dimensional constant rank $m$ subspace of $\Symm(V)$. Then, given a vector $u$ in $V^\times$, the subspace $\A_u$ is totally isotropic for $\M$ provided that $|K|\geq m+1$ and $K$ has characteristic different from $2$.

\end{lemma}

\begin{proof}
Let $w$ be an element of $\A_u^\times$. We aim to show that $w\in I(\M)^\times$. Now since $\epsilon_u(\M)$ vanishes on $w$, the symmetry of our forms implies that
$\epsilon_w(\M)$ vanishes on $u$. Then, since $u\neq 0$, $\epsilon_w(\M)\neq V^*$.

Now we have the general equality
\[
\dim \epsilon_w(\M)+\dim \M_w=\dim \M=n
\]
and since we know from above that $\dim \epsilon_w(\M)<n$, we deduce that $\M_w\neq 0$. It follows that there is some element $f$ of $\M^\times$ with $w\in \rad f$.

But $\M$ is a constant rank $m$ subspace of $\Symm(V)$, and since we are assuming that $|K|\geq m+1$, Theorem \ref{left_radical_right_radical} implies that
$\rad f$ is totally isotropic for $\M$. Thus $w\in I(\M)^\times$ and hence $\A_u$ is contained in $I(\M)$.

Finally, applying the assumption that $K$ has characteristic different from 2, Lemma
\ref{totally_isotropic_in_subspace} implies that $\A_u$ is totally isotropic for $\M$.
\end{proof}

The next lemma is elementary.

\begin{lemma} \label{A_w_property}

Let $\M$ be a subspace of $\Symm(V)$. Let $W$ be a subspace of $V$ that is totally isotropic for $\M$. Then $W$ is contained in $\A_w$ for each element $w$ of $W$.

\end{lemma}

\begin{proof}
Since $W$ is totally isotropic for $\M$,
\[
f(w,u)=0
\]
for all $u$ in $W$ and all $f$ in $\M$. This implies that $W\leq \A_w$, as required.
\end{proof}

We proceed to determine
$\dim \A_u$.

\begin{lemma} \label{dimension_of_A_u}
Let $\M$ be an $n$-dimensional constant rank $m$ subspace of $\Symm(V)$
and let $u$ be an element of $V^\times$.  Then $\dim \A_u=\dim \M_u$.

Suppose furthermore that $|K|\geq m+1$ and $K$ has characteristic different from $2$. Then
$\A_u$ is non-zero if and only if $u\in I(\M)^\times$.
\end{lemma}

\begin{proof}
We have
\[
\dim \epsilon_u(\M)=\dim \M-\dim \M_u=n-\dim \M_u.
\]
Duality theory shows that the subspace of $V$ annihilated by $\epsilon_u(\M)$ has dimension $n-\dim \epsilon_u(\M)$ and this number equals $\dim \M_u$ by our equality above. Thus, $\dim \A_u=\dim \M_u$.

Suppose that $K$ has the properties described above and $\A_u$ is non-zero. Then $\A_u$ is totally isotropic for $\M$, by 
Lemma \ref{key_lemma_isotropic}. Thus $u\in I(\M)^\times$.
Conversely, suppose that $u\in I(\M)^\times$. Then $u\in \A_u$ by Lemma \ref{A_w_property}
and hence $\A_u$ is non-zero.
\end{proof}

The next in our sequence of lemmas enables us to partition $I(\M)$.

\begin{lemma} \label{partition_theorem}
Let $\M$ be an $n$-dimensional constant rank $m$ subspace of $\Symm(V)$. Suppose that $|K|\geq m+1$ and $K$ has characteristic different from $2$. Let
$u$ be an element of $I(\M)^\times$. Then we have $\A_w=\A_u$ for all elements $w$ of $\A_u^\times $.
\end{lemma}

\begin{proof}
Let $w$ be an element of $\A_u^\times $. Since $\A_u$ is totally isotropic for $\M$ by Lemma \ref{key_lemma_isotropic}, $\A_u\leq \A_w$ by Lemma \ref{A_w_property}. 

Furthermore, since $\epsilon_u(\M)$ annihilates $w$, $\epsilon_w(\M)$ annihilates $u$
by symmetry and thus $u\in \A_w$. Since $\A_w$ is also totally isotropic for $\M$,
$\A_w\leq \A_u$, again by Lemma \ref{A_w_property}. Thus $\A_w=\A_u$, as stated.
\end{proof}

\begin{corollary} \label{statement_of_partition_property}

Let $\M$ be an $n$-dimensional constant rank $m$ subspace of $\Symm(V)$. Suppose  that $|K|\geq m+1$ and $K$ has characteristic different from $2$. 
Then $I(\M)$ is the union of the subspaces $\A_u$, where $u$ runs through the  elements of $I(\M)^\times$. Two subspaces $\A_u$ and $\A_w$ are either identical or
$\A_u\cap \A_w=0$.

\end{corollary}

\begin{proof}
Let $u$ be an element $I(\M)^\times$. We have seen that $\A_u$ is non-zero and contained in $I(\M)$ by Lemma \ref{key_lemma_isotropic}. Since $u\in \A_u$, $I(\M)$ is the union of subspaces of type $\A_u$.

Consider now a second subspace $\A_w$. Suppose $v$ is a non-zero vector in
$\A_u\cap \A_w$. Then $v\in \A_u$ and thus $\A_v=\A_u$ by Lemma \ref{partition_theorem}.
Likewise, $v\in \A_w$ and thus $\A_v=\A_w$. Therefore, $\A_u=\A_w$. 
\end{proof}

We shall assume for the rest of this section that the underlying field $K$ is finite of
odd characteristic.
Corollary \ref{statement_of_partition_property} implies that in this case
$I(\M)$ is the union of a finite number, $r$, say, of subspaces $\A_w$ which intersect trivially pairwise. Thus $I(\M)^\times$ is a disjoint union of $r$ subsets $\A_i^\times$, where
$\A_i=\A_{u_i}$, $1\leq i\leq r$. 

If $\A_u$ is one of the $\A_i$, we have $\dim \A_u=\dim \M_u=d(u)$, say, where
$d(u)>0$, by Lemma \ref{dimension_of_A_u}. Thus if we set $d_i=d(u_i)$, we have
\[
|I(\M)^\times|=\sum_{i=1}^r (q^{d_i}-1).
\]
Retaining this notation, we have the following  result.

\begin{theorem} \label{key_counting_theorem}
Let $\M$ be an $n$-dimensional constant rank $m$ subspace of $\Symm(V)$ and $K=\mathbb{F}_q$, where $q$ is odd and at least $m+1$.
Then if
\[
I(\M)^\times=\bigcup _{i=1}^r \A_i^\times,
\]
where $|\A_i^\times|=q^{d_i}-1$, we have
\[
\sum_{i=1}^r (q^{d_i}-1)^2=(q^n-1)(q^{n-m}-1).
\]
\end{theorem}

\begin{proof}
Lemma \ref{constant_rank_dimension_lemma} shows that
\[
\sum_{u\neq 0} (q^{d(u)}-1)=(q^n-1)(q^{n-m}-1),
\]
where $d(u)=\dim \M_u$ and we need only to sum over the elements
$u$ of $I(\M)^\times$. We evaluate the sum on the left by counting over the subsets
$\A_i^\times$ which partition $I(\M)^\times$.

The subset contains $\A_i^\times$ contains $q^{d_i}-1$ elements $u$, for each of which
$d(u)=d_i$. Thus the contribution of the elements of $\A_i^\times$ to the sum is
$(q^{d_i}-1)^2$. Summing over all $i$, we obtain
\[
\sum_{i=1}^r (q^{d_i}-1)^2=(q^n-1)(q^{n-m}-1),
\]
as required.
\end{proof}

The next result shows that $I(\M)$ cannot be a subspace in the circumstances of
Theorem \ref{key_counting_theorem}. 

\begin{lemma} \label{r=1_case}

We cannot have $r=1$ in Theorem \ref{key_counting_theorem}.

\end{lemma}

\begin{proof}
Suppose if possible that $r=1$ in Theorem \ref{key_counting_theorem}. Then setting
$d=d_1$, we obtain
\[
(q^d-1)^2=(q^n-1)(q^{n-m}-1).
\]

A simple substitution in the formula above shows that $d$ cannot equal $n-m$. Thus $d>n-m$.
Expansion of the formula above yields
\[
q^{2d}-2q^d=q^{2n-m}-q^n-q^{n-m}.
\]
Since we are assuming that $q$ is odd, the power of $q$ dividing the left hand side is $q^d$, and the power of $q$ dividing the right hand side is $q^{n-m}$. 

This implies that $d=n-m$, a case we have already eliminated. Consequently, $r=1$ is impossible.
\end{proof}

We turn to the proof of the main theorem of this section.

\begin{theorem} \label{the_main_dimension_bound}

Let $\M$ be a constant rank $m$ subspace of $\Symm(V)$ and  $K=\mathbb{F}_q$, where $q$ is odd and at least $m+1$. Then if $m\leq 2n/3$, we have
$\dim \M<n$.

\end{theorem}

\begin{proof}

Suppose by way of contradiction that $\dim \M=n$. We may then apply Theorem 
\ref{key_counting_theorem}. Let the distinct dimensions $d_i$ that occur be $e_1$, \dots, $e_t$, where
\[
n-m\leq e_1<\ldots <e_t.
\]
Suppose that dimension $e_i$ occurs with multiplicity $c_i$ (so that $r=c_1+ \cdots +c_t$). Then we have 
\[
\sum_{i=1}^t c_i(q^{e_i}-1)^2=(q^n-1)(q^{n-m}-1), \quad |I(\M)^\times|=\sum_{i=1}^t c_i(q^{e_i}-1).
\]

Let us first show that we cannot have $t=1$ and $e_1=n-m$. For if we take
$t=1$ and $e_1=n-m$, we obtain
\[
c_1(q^{n-m}-1)^2=(q^n-1)(q^{n-m}-1), \quad |I(\M)^\times|=c_1(q^{n-m}-1).
\]
These equations imply that $|I(\M)^\times|=q^n-1$, a clear contradiction since it means that every vector in $V$ is isotropic for $\M$. Thus our statement is established.

Expanding the first equation involving $(q^n-1)(q^{n-m}-1)$, we obtain
\[
\sum_{i=1}^t c_i(q^{2e_i}-2q^{e_i}+1)=q^{2n-m}-q^n-q^{n-m}+1.
\]
Now as $e_i\geq n-m$ for all $i$, $q^{n-m}$ divides $q^{2e_i}$, $q^{e_i}$, $q^{2n-m}$
and $q^n$. It follows that $q^{n-m}$ divides
\[
(\sum_{i=1}^t c_i)-1=r-1.
\]
Since we have shown that $r>1$ in Lemma \ref{r=1_case}, we have then
\[
(\sum_{i=1}^t c_i)=\lambda q^{n-m}+1
\]
for  some positive integer $\lambda$. 

Now as we already know that all $e_i$ satisfy $e_i\geq n-m$, and moreover that not all
$e_i$ equal $n-m$, we deduce that
\[
|I(\M)^\times|>(\lambda q^{n-m}+1)(q^{n-m}-1)
\]
and thus $|I(\M)^\times|>q^{2(n-m)}-1$. 

Now we can certainly assume that $m$ is even, since if $m$ is odd, a constant rank
$m$ subspace of $\Symm(V)$ has dimension at most $m$, and thus our theorem is trivially true. See, for example, Corollary 3  of \cite{DGS}.

We therefore set $m=2k$, where $k$ is a positive integer. Then since $\dim \M=n$, Theorem 5 of \cite{DGS} implies that
\[
|I(\M)^\times|=(A-B)q^{-k},
\]
where $A$ is the number of elements of Witt index $k$ in $\M$, $B$ is the number of elements of Witt index $k-1$ in $\M$, and $A+B=q^n-1$. Thus, since $|I(\M)^\times|$ is an integer, we have $|I(\M)^\times|\leq q^{n-k}-1$. 

If we now compare this upper bound with our earlier lower bound for $|I(\M)^\times|$, we obtain
\[
q^{2n-2m}-1<q^{n-k}-1
\]
and deduce that $2n-2m<n-k=n-m/2$. This yields $n<3m/2$, and contradicts our original hypothesis that $n\geq 3m/2$. We therefore conclude that $\dim \M<n$, as required.
\end{proof}

\section{Dimension bounds for subspaces of alternating forms}

\noindent We begin this section by proving a sharpened version of Lemma \ref{M_u_dimension_bound} for subspaces of $\Alt(V)$. 

\begin{lemma} \label{M_u_alternating_bound} 

Let $\M$ be a subspace of $\Alt(V)$ and let $u$ be an element of $V$. Then we have
$\dim \M_u\geq \dim \M-(n-1)$.

\end{lemma}

\begin{proof}
The elements of $\M$ are alternating by hypothesis and hence   $\epsilon_u(\M)$ vanishes on $u$. It follows that $\epsilon_u(\M)\neq V^*$. Thus $\dim \epsilon_u(\M)\leq n-1$ and the inequality
for $\dim \M_u$ is a consequence of the rank-nullity theorem.
\end{proof}

The next result enables us to establish dimension bounds for subspaces $\M$ of $\Alt(V)$ by induction on $|\rank(\M)|$.

\begin{theorem} \label{alternating_dimension_bound}

Let $\M$ be a non-zero subspace of $\Alt(V)$ and let $m$ be the largest integer in $\rank(\M)$. Suppose that $|K|\geq m+1$ and for each element $w$ of $V$, $\M_w$ contains an element of rank $m$.

Then there exist elements $u$ and $v$ in $V$ such that 
$\M_u\cap \M_v$ contains no elements of rank $m$ and the inequality 
\[
\dim \M\leq 2m-1+\dim (\M_u\cap \M_v)
\]
holds. Furthermore, if $U$ is the two-dimensional subspace of $V$ spanned by $u$ and $v$, and $W$ is a complement of $U$ in $V$, we may identify $\M_u\cap \M_v$ with a subspace of $\Alt(W)$.

\end{theorem}

\begin{proof}
Since $\M$ is non-zero, there must exist $u$ and $v$ in $V$ such that $g(u,v)\neq 0$ for some $g\in \M$. We note that $u$ and $v$ are necessarily linear independent.

We will now show that $\M_u\cap \M_v$ contains no elements of rank $m$. For suppose that some form $f$ of rank $m$ is in
$\M_u\cap \M_v$. Then $u$, $v$ are in $\rad f$ and hence as $|K|\geq m+1$, Theorem
\ref{left_radical_right_radical} implies that $h(u,v)=0$ for all $h\in\M$. This contradicts our earlier assertion about the existence of $g$ in $\M$ with $g(u,v)\neq 0$. We deduce that $\M_u\cap \M_v$ contains no elements of rank $m$, as asserted.

We note that each element $\phi$, say, of the subspace $\M_u+\M_v$ of $\M$ satisfies
$\phi(u,v)=0$, since $u$ is in the radical of each element of $\M_u$ and
$v$ is in the radical of each element of $\M_v$. In view of our earlier statement about the existence of $g$ in $\M$ with $g(u,v)\neq 0$, we see that $\M_u+\M_v\neq \M$ and hence $\dim(\M_u+\M_v)<\dim \M$. It follows that
\[
\dim \M_u+\dim \M_v<\dim \M+\dim(\M_u\cap \M_v).
\]

Now since $\M_u$ and $\M_v$ both contain elements of rank $m$, Lemma \ref{improved_M_u_dimension_bound} shows that
\[
\dim \M-m\leq \dim \M_u, \quad \dim \M-m\leq \dim \M_v.
\]
Thus we have
\[
2(\dim \M-m)<\dim \M+\dim(\M_u\cap \M_v)
\]
and consequently
the inequality 
\[
\dim \M\leq 2m-1+\dim (\M_u\cap \M_v)
\]
is valid.

Finally, $U$ is contained in the radical of each element of $\M_u\cap \M_v$. It follows that if $W$ is a complement of $U$ in $V$, each element of $\M_u\cap \M_v$ 
is determined by its restriction to $W\times W$ and thus we may identify $\M_u\cap \M_v$ with a subspace of $\Alt(W)$.
\end{proof}

\begin{theorem} \label{alternating_constant_rank_bound}

Let $\M$ be a non-zero constant rank $m$ subspace of $\Alt(V)$.  Then if $|K|\geq m+1$, we have $\dim \M\leq \max\,(n-1, 2m-1)$. 
\end{theorem}

\begin{proof}
Suppose that for some $w$ in $V$, $\M_w=0$. Then it follows from Lemma \ref{M_u_alternating_bound} that $\dim \M\leq n-1$, and there is nothing more to prove.

We may therefore assume that $\M_w\neq 0$ for all $w\in V$. It follows from Theorem
\ref{alternating_dimension_bound} that there exist elements $u$ and $v$ in $V$ such that 
$\M_u\cap \M_v$ contains no elements of rank $m$. In view of the constant rank $m$ hypothesis, this implies that $\M_u\cap \M_v=0$ and hence
\[
\dim \M\leq 2m-1,
\]
by Theorem \ref{alternating_dimension_bound}, as asserted.
\end{proof}

Our next step is to prove a version of Theorem \ref{alternating_constant_rank_bound}
for subspaces $\M$ of $\Alt(V)$ with $|\rank(\M)|>1$. 

\begin{theorem} \label{alternating_several_ranks_bound}

Let $\M$ be a non-zero 
subspace of $\Alt(V)$, let $m$ be the largest integer in $\rank(\M)$ and let
$r=|\rank(\M)|$. Then if
$m\leq \lfloor n/2\rfloor$ 
and $|K|\geq m+1$, we have
\[
\dim \M\leq rn-\frac{r(r+1)}{2}.
\]
\end{theorem}

\begin{proof}
We proceed by induction on $r$ and may assume that $r\geq 1$.

Suppose that there is an element $w$ of $V$ such that $\M_w$ contains no element of rank $m$. Let $U$ be the one-dimensional subspace of $V$ 
spanned by $w$ and let $U'$ be a complement of $U$ in $V$. Since $U$ is contained
in the radical of each element of $\M_w$, we may identify $\M_w$ with a subspace of
 $\Alt(U')$. We also have $|\rank(\M_w)|\leq r-1$, since $\M_w$ contains no element
 of rank $m$.
 
 Let $m'$ be the largest integer in $\rank(\M_w)$. We have $m'\leq m-2$, since 
 $m'<m$ and each element of $\Alt(V)$ has even rank. It follows then by a simple inspection that $m'\leq \lfloor (n-1)/2\rfloor$. By induction, we have
 \[
 \dim \M_w\leq (r-1)(n-1)-\frac{(r-1)r}{2}.
\]
 Then since $\dim \M\leq \dim \M_w+n-1$, by Lemma \ref{M_u_alternating_bound}, we obtain that
\[
 \dim \M\leq (r-1)(n-1)-\frac{(r-1)r}{2}+n-1= rn-\frac{r(r+1)}{2}.
\] 
 We have thus completed the induction step in this case. 
 
 We can now assume that for each element $w$ of $V$, $\M_w$ contains an element of rank $m$. It follows from Theorem \ref{alternating_dimension_bound} that
there exist elements $u$ and $v$ in $V$ such that 
$\M_u\cap \M_v$ contains no elements of rank $m$ and the inequality 
\[
\dim \M\leq 2m-1+\dim (\M_u\cap \M_v)
\]
holds. Moreover, we may identify $\M_u\cap \M_v$ with a subspace of $\Alt(W)$, where 
$W$ is a subspace of $V$ of dimension $n-2$. 

We clearly have $|\rank(\M_u\cap \M_v)|\leq r-1$. Moreover, if $m''$ is the largest
integer in $\rank(\M_u\cap \M_v)$, we have $m''\leq m-2$, as above. Then it is easy to verify that $m''\leq  \lfloor (n-2)/2\rfloor$. Thus, by induction
\[
\dim (\M_u\cap \M_v)\leq (r-1)(n-2)-\frac{(r-1)r}{2}
\]
and hence by the inequality in the paragraph above, 
\[
\dim \M\leq (r-1)(n-2)-\frac{(r-1)r}{2}+2m-1.
\]
 
Our hypothesis that $m\leq \lfloor n/2\rfloor$ implies that $2m-1\leq n-1$. We thus have
 \[
\dim \M\leq (r-1)(n-2)-\frac{(r-1)r}{2}+n-1=rn-\frac{(r^2+3r-2)}{2}.
\]
It is straightforward to verify that $r^2+3r-2\geq r(r+1)$ when $r\geq 1$.
Thus
\[
\dim \M\leq rn-\frac{r(r+1)}{2}
\]
holds for $r\geq 1$ and we have  completed the induction step in this second case. Therefore, the dimension bound holds for all $r$, as required. 
\end{proof}

We remark that it easy to see that this dimension bound is optimal in certain cases.
On the other hand, the bound does not hold without some restriction on $m$. For example, suppose that $n=2k+1$ is odd and $K$ is a finite field. Then it is possible to construct a subspace $\M$ of $\Alt(V)$ such that $\dim \M=(k-s+1)n$ and
$\rank \M=\{2s, 2s+2, \ldots, 2k\}$, where the integer $s$ satisfies $1\leq s\leq k$. Note that when $s=1$, $\M$ is $\Alt(V)$. 

\section{Constant rank subspaces of alternating forms over finite fields}

\noindent The general results we have obtained in Section 3 can be used to deduce
reasonably precise information about $n$-dimensional constant rank subspaces of $\Alt(V)$ when we work over sufficiently large finite fields. Our first theorem
gives the basic details.

\begin{theorem} \label{spread_theorem}
 Let
$\M$ be an $n$-dimensional constant rank $m$ subspace of $\Alt(V)$ and let
$K=\mathbb{F}_q$, where
$q\geq m+1$. Let $R_1$, \dots, $R_t$ be the different subspaces of dimension $n-m$ in $V$ that occur as the radicals of the elements of $\M^\times$. Then these subspaces form a spread of $V$. As a consequence, $t=(q^n-1)/(q^{n-m}-1)$ and hence $n-m$ divides $n$.

Furthermore, if $\M_i$ is the subspace of $\M$ consisting of those elements of $\M$ whose radical contains $R_i$, $1\leq i\leq t$, these $t$ subspaces form a spread of $\M$. 

\end{theorem}

\begin{proof}
Let $u$ be any element of $V^\times$. Since $\M$ consists of alternating bilinear
forms and $\dim \M=n$, Lemma \ref{M_u_alternating_bound} implies that $\M_u$ in non-zero. Thus, since $\M$ is a constant rank $m$ subspace and $q\geq m+1$, Lemma \ref{improved_M_u_dimension_bound} implies that $d(u)=\dim \M_u\geq n-m$.
Lemma \ref{constant_rank_dimension_lemma}  then yields that
\[
\sum_{u\neq 0}(q^{d(u)}-1)=(q^n-1)(q^{n-m}-1).
\]
Since $d(u)\geq n-m$, we see that the left hand side above is at least $(q^n-1)(q^{n-m}-1)$, and can only equal $(q^n-1)(q^{n-m}-1)$ if $d(u)=n-m$ for all $u\neq 0$. We deduce that  $d(u)=n-m$ for all non-zero $u$.

It now follows from Lemma \ref{improved_M_u_dimension_bound} that each element of $\M_u^\times $ has the same radical, $R$, say, where $R$ depends on $u$. Let $w$ be an element of $V$ not in $R$ and let $S$ be the common radical of the  elements in 
$\M_w^\times$. Note that since $w\in S$, $R\neq S$.

We claim next that $\M_u\cap \M_w=0$ and $R\cap S=0$. For suppose that $f$ is a non-zero element of  $\M_u\cap \M_w$. Then $u\in \rad f=R$ and $w\in \rad f=R$. This is clearly absurd, and we deduce that $\M_u\cap \M_w=0$, as claimed. 

Similarly, suppose that $v$ is a non-zero element of $R\cap S$. If $g$ is any  element of $\M_u^\times$, $v\in \rad g=R$. But equally, if $h$ is any  element of $\M_w^\times$, $v\in \rad h=S$. But we also have $g\in \M_v$ and $h\in \M_v$. This implies that $g$ and $h$ have the same radical, as all elements of $\M_v$ have the same radical. This gives the contradiction that $R=S$. We deduce that $R\cap S=0$.

It follows that if $R_1$, \dots, $R_t$ are the different subspaces of dimension $n-m$ in $V$ that occur as the radicals of the elements of $\M^\times$, these subspaces form a spread of $V$.  We deduce that $t=(q^n-1)/(q^{n-m}-1)$ and hence $n-m$ divides $n$.
Likewise, if $\M_i$ is the subspace of $\M$ consisting of those elements of $\M$ whose radical contains $R_i$, $1\leq i\leq t$, these $t$ subspaces form a spread of $\M$.
\end{proof}

We turn now to considering to what extent the converse of Theorem \ref{spread_theorem} holds: in other words, if $n-m$ divides $n$ (and $m$ is even), is there a constant rank $m$ subspace of $\Alt(V)$ of dimension $n$? As we shall see, the answer is no in general, but there are many cases for which the converse is true.

Our starting point is the following construction. Suppose that $U$ is a vector space of odd dimension $k>1$  over an arbitrary finite field. Then there is a 
$k$-dimensional constant rank $k-1$ subspace, $\N$, say, of $\Alt(U)$. The spread associated
with $\N$ is the trivial one, consisting of all one-dimensional subspaces of $U$. See, for example, Theorem 7 of \cite{DGo}.

If we take the field to be $\mathbb{F}_{q^t}$, and consider $U$ and $\N$ as vector spaces of dimension $n=kt$ over $\mathbb{F}_{q}$, the trace map from $\mathbb{F}_{q^t}$ to $\mathbb{F}_{q}$ enables us derive a constant rank $m=(k-1)t$ subspace $\M$, say, of $\Alt(V)$, where $\dim \M=\dim V=n=kt$, and $\M$ and $V$ are vector spaces
over $\mathbb{F}_{q}$. We have therefore the following result.

\begin{theorem} \label{constant_rank_construction}

Let $K$ be a finite field and let
$n\geq 3$ be a positive integer that is not a power of $2$. Let $k>1$ be an odd divisor of $n$ and let $m=n(k-1)/k$. Then $\Alt(V)$ contains an $n$-dimensional constant rank $m$ 
subspace. In particular, if $n$ is odd, this construction holds for any divisor of $n$ greater than $1$.

\end{theorem}

We mention the following non-existence criterion for constant rank 
subspaces. Suppose that $n=4k$, where $k$ is a positive integer, and 
 $|K|\geq 2k+1$.  Then $\Alt(V)$ does not contain a constant rank $2k$ subspace of dimension $n$. This is an immediate consequence of Theorem \ref{alternating_constant_rank_bound}.

The remainder of this section is devoted to improving Theorem \ref{alternating_constant_rank_bound} in the case of finite fields. Initially, however, we work over general fields. 

\begin{definition} \label{R(M)_definition}
Let $\M$ be a subspace of $\Alt(V)$. We set
\[
V(\M)=\{v\in V: \M_v\neq 0\}.
\]
\end{definition}

Since we are dealing with alternating bilinear forms, $V(\M)$ is identical with
the subspaces $V(\M)^L$ and $V(\M)^R$ introduced in Definition \ref{V^L_M_definition}.
Lemmas \ref{V(M)_is_a_subspace} and \ref{left_orthogonal_right_orthogonal} show that
if $\M$ is a constant rank $m$ subspace
of $\Alt(V)$, then $V(\M)$ is a subspace of $V$ that is totally isotropic for $\M$ provided that $\dim \M\geq 2m+1$ and $|K|\geq m+1$. We will show below that an improved result holds if $K$ is a finite field.

We begin our analysis by looking at an extreme case.

\begin{lemma} \label{gow_quinlan_lemma}
Let $\M$ be a constant rank $m$ subspace of $\Alt(V)$ and  $K=\mathbb{F}_q$.
Suppose that $V(\M)=\rad f$ for some $f\in \M^\times$. Then $\dim \M\leq m/2$.
\end{lemma}

\begin{proof}
We set $R=\rad f$. Our hypothesis implies that each element of $\M^\times$ has radical
$R$. We may then identify $\M$ with a constant rank $m$ subspace of $\Alt(V/R)$. Since $\dim (V/R)=m$, we have $\dim \M\leq m/2$ by Lemma 3 of \cite{GQ}.
\end{proof}

\begin{lemma} \label{R(M)_is_a_subspace}

Let $\M$ be a constant rank $m$ subspace of $\Alt(V)$ and  $K=\mathbb{F}_q$.
Suppose that $q\geq m+1$, $m>2$ and $\dim \M\geq 2m-1$. Then $V(\M)$ is a subspace of $V$ that is totally isotropic for $\M$. 

\end{lemma}

\begin{proof}
Let $u$ be an element of $V(\M)$. Lemma \ref{improved_M_u_dimension_bound} shows that
$\dim \M_u\geq \dim \M-m$ and moreover, if $\dim \M_u=\dim \M-m$, all elements of
$\M_u^\times$ have the same radical. Suppose then that we have $\dim \M_u=\dim \M-m$.
Lemma \ref{gow_quinlan_lemma} applied to $\M_u$ implies that $\dim \M_u\leq m/2$ and 
this in turn implies that $\dim \M\leq 3m/2$. But we are already assuming that 
$\dim \M\geq 2m-1$ and hence $2m-1\leq 3m/2$. This last inequality implies that $m\leq 2$, which possibility is excluded by hypothesis.

We have thus established that $\dim \M_u\geq \dim \M-m+1$ when $u\in V(\M)$. We intend to use this inequality to show that $V(\M)$ is a subspace. It is clear that $V(\M)$ is closed under scalar multiplication. Given elements $u$ and $v$ in $V(\M)$ it suffices then to show that $u+v\in V(\M)$. As in the proof of Lemma \ref{V(M)_is_a_subspace}, this follows if we can show that $\M_u\cap \M_v\neq 0$.

We have 
\[
\dim(\M_u+\M_v)=\dim \M_u+\dim \M_v-\dim (\M_u\cap \M_v)\leq \dim \M.
\]
Since we have established that $\dim \M_u, \dim \M_v\geq \dim \M-m+1$, we obtain
\[
2(\dim \M-m+1)-\dim (\M_u\cap \M_v)\leq \dim \M.
\]
Given that we are assuming that $\dim \M\geq 2m-1$, this inequality leads to
\[
\dim (\M_u\cap \M_v)\geq \dim \M-2m+2\geq 1
\]
and we have thus proved that the intersection is non-zero, as required. Thus
$V(\M)$ is a subspace. 

Finally, we need to show that $V(\M)$ is totally isotropic for $\M$. To achieve this, we must
prove that $g(u,v)=0$ for all $u$ and $v$ in $V(\M)$, and all $g$ in $\M$. Now since we have proved that $\M_u\cap \M_v\neq 0$, $u$ and $v$ are both contained in $\rad f$, where $f$ is any non-zero element of $\M_u\cap \M_v$.
Since $\rad f$ is totally isotropic for $\M$ by Theorem \ref{left_radical_right_radical}, we see that $g(u,v)=0$ for $g$ in $\M$. This establishes the total isotropy of $V(\M)$.
\end{proof}

\begin{lemma} \label{dimension_of_R(M)}

Let $\M$ be a constant rank $m=2k$ subspace of $\Alt(V)$ and
$K=\mathbb{F}_q$. Suppose that $V(\M)$ is a subspace of $V$ and not all elements of $\M^\times$ have the same radical. Then if $\dim \M> n-k$, we have $\dim V(\M)\geq n-m+2$.

\end{lemma}

\begin{proof}
Since $V(\M)$ contains all radicals of elements of $\M^\times$, and we are assuming that not all these radicals are the same, $\dim V(\M)\geq n-m+1$. 

Suppose now by way of contradiction that $\dim V(\M)=n-m+1$. Since $V(\M)$ contains the radical of each element of $\M^\times$, and each radical has codimension 1 in
$V(\M)$, the number of different radicals is at most $(q^{n-m+1}-1)/(q-1)$. 

On the other hand, let $S$ be the radical of some element of $\M^\times$, and let
\[
\M_S=\{ g\in \M: S\leq \rad g\}.
\]
$\M_S$ is a subspace of $\M$ and Lemma \ref{gow_quinlan_lemma} implies that
$\dim \M_S\leq m/2$. If $S'=\rad h$, where $h\in\M^\times$, and $S\neq S'$, we clearly have $\M_S\cap \M_{S'}=0$. It follows that the subspaces $\M_S$ form a partition of $\M$. Thus if $\dim \M=d$, the number of different radicals is at least
$(q^d-1)/(q^k-1)$, where $m=2k$.

We deduce that the inequality
\[
\frac{q^d-1}{q^k-1}\leq \frac{q^{n-m+1}-1}{q-1}
\]
holds. This leads to the inequality 
\[
q^{d+1}\leq q^d+q^{n-k+1}-q^k-q^{n-2k+1}+q.
\]
It is clear, however, that this inequality cannot hold if $d>n-k$. 
\end{proof}

\begin{lemma} \label{R(M)_inequality}

Let $\M$ be a constant rank $m$ subspace of $\Alt(V)$ and  $K=\mathbb{F}_q$.
Suppose that $q\geq m+1$, $m>2$ and $\dim \M\geq 2m-1$. Then
we have
\[
\dim V(\M)\leq \dim \M_u+n-\dim \M
\]
for each element $u$ of $V(\M)$.

\end{lemma}

\begin{proof}
Lemma \ref{R(M)_is_a_subspace} shows that $V(\M)$ is a subspace that is totally isotropic for $\M$. The rest of the proof follows from the argument used to prove Lemma \ref{dimension_estimate_in_terms_of_V}.
\end{proof}

\begin{theorem} \label{new_alternating_dimension_bound}

Let $\M$ be a constant rank $m$ subspace of $\Alt(V)$, where $4\leq m\leq
\lfloor n/2\rfloor$ and $K=\mathbb{F}_q$. Then if $q\geq m+1$, we have $\dim \M\leq n-2$.

\end{theorem} 

\begin{proof}
Suppose if possible that $\dim \M=n-1$. Then since we are assuming that 
$m\leq \lfloor n/2\rfloor$, we have $\dim \M\geq 2m-1$ and hence Lemma \ref{R(M)_is_a_subspace} implies that $V(\M)$ is a subspace of $V$ that is totally isotropic for $\M$. We assert furthermore that the elements of $\M^\times$ do not all have the same radical. This follows from Lemma \ref{gow_quinlan_lemma}, since $\dim \M>m/2$ in our circumstances.

Now since we are assuming that $\dim \M=n-1$, we certainly have
$\dim \M>n-m/2$ and hence $\dim V(\M)\geq n-m+2$ by Lemma \ref{dimension_of_R(M)}. Lemma \ref{R(M)_inequality} in turn implies that 
\[
\dim \M_u\geq n-m+1
\]
for all $u$ in $V(\M)$.

Lemma \ref{constant_rank_dimension_lemma} shows that we have the equality
\[
\sum_{u\neq 0} (q^{d(u)}-1)=(q^{n-1}-1)(q^{n-m}-1),
\]
where $d(u)=\dim \M_u$. The sum clearly takes place over the non-zero elements
$u$ of $V(\M)$. Thus, if we set $d=\dim V(\M)$, we have
\[
 (\sum_u q^{d(u)})-q^d+1=(q^{n-1}-1)(q^{n-m}-1).
\]
Consequently, 
\[
(\sum_u q^{d(u)})-q^d=q^{2n-m-1}-q^{n-1}-q^{n-m}.
\]

Since we have shown above that $d\geq n-m+2$ and $d(u)\geq n-m+1$, we see that the power of $q$ dividing the left hand side of the sum above is at least $q^{n-m+1}$, whereas the power of $q$ dividing the right hand side is exactly $q^{n-m}$. This is a contradiction, and we deduce that $\dim \M\leq n-2$.
\end{proof}

We note that the assumption that $m\geq 4$ is essential to this theorem, since for any field $K$, there is a constant rank 2 subspace of $\Alt(V)$ of dimension $n-1$.
On the other hand, if $K$ is finite, $\Alt(V)$ contains a constant rank 4 subspace of dimension $n-2$ whenever $n$ is even and at least 4.

\section{Dimension bounds for subspaces of bilinear forms}

\noindent In this final section, we establish analogues of earlier results which apply to bilinear forms in general.

\begin{theorem} \label{bilinear_dimension_bound}

Let $\M$ be a non-zero subspace of $\Bil(V)$ that is not contained in
$\Alt(V)$ and let $m$ be the largest integer in $\rank(\M)$. Suppose that $|K|\geq m+1$. Suppose also that for each element $w$ of $V$, both $\M_w^L$ and $\M_w^R$ contain  elements of rank $m$.

Then there exists an element $u$ in $V$ such that 
$\M_u^L\cap \M_u^R$ contains no elements of rank $m$ and the inequality 
\[
\dim \M\leq 2m-1+\dim (\M_u^L\cap \M_u^R)
\]
holds. Furthermore, if $U$ is the one-dimensional subspace of $V$ spanned by $u$ and $W$ is a complement of $U$ in $V$, we may identify $\M_u^L\cap \M_u^R$ with a subspace of $\Bil(W)$.

\end{theorem}

\begin{proof}
Since $\M$ is not contained in $\Alt(V)$, there exists $u$ in $V$ such that $g(u,u)\neq 0$ for some $g\in \M$.

We will now show that $\M_u^L\cap \M_u^R$ contains no elements of rank $m$. For suppose that some form $f$ of rank $m$ is in
$\M_u^L\cap \M_u^R$. Then $u$ is in both $\rad_L f$ and $\rad_R f$, and hence as $|K|\geq m+1$, Theorem
\ref{left_radical_right_radical} implies that $h(u,u)=0$ for all $h\in\M$. This contradicts our earlier assertion about the existence of $g$ in $\M$ with $g(u,u)\neq 0$. We deduce that $\M_u^L\cap \M_u^R$ contains no elements of rank $m$, as asserted.

We note that each element $\phi$, say, of the subspace $\M_u^L+\M_u^R$ of $\M$ satisfies
$\phi(u,u)=0$, since $u$ is in the left radical of each element of $\M_u^L$ 
and also in the right radical of each element of $\M_u^R$. In view of our earlier statement about the existence of $g$ in $\M$ with $g(u,u)\neq 0$, we see that $\M_u^L+\M_u^R\neq \M$ and hence $\dim(\M_u^L+\M_u^R)<\dim \M$. It follows that
\[
\dim \M_u^L+\dim \M_u^R<\dim \M+\dim(\M_u^L\cap \M_u^R).
\]

Now since $\M_u^L$ and $\M_u^R$ both contain elements of rank $m$, Lemma \ref{improved_M_u_dimension_bound} shows that
\[
\dim \M-m\leq \dim \M_u^L, \quad \dim \M-m\leq \dim \M_u^R.
\]
Thus we have
\[
2(\dim \M-m)<\dim \M+\dim(\M_u^L\cap \M_u^R)
\]
and consequently
the inequality 
\[
\dim \M<2m+\dim (\M_u^L\cap \M_u^R)
\]
is valid. We therefore have $\dim \M\leq 2m-1+\dim (\M_u^L\cap \M_u^R)$.

Finally, $U$ is contained in the left and right radical of each element of $\M_u^L\cap \M_u^R$. It follows that if $W$ is a complement of $U$ in $V$, each element of $\M_u^L\cap \M_u^R$ 
is determined by its restriction to $W\times W$ and thus we may identify $\M_u^L\cap \M_u^R$ with a subspace of $\Bil(W)$.
\end{proof}

\begin{theorem} \label{bilinear_constant_rank_bound}

Let $\M$ be a non-zero constant rank $m$ subspace of $\Bil(V)$. Then if $|K|\geq m+1$, we have $\dim \M\leq \max\,(n,2m-1)$. 
\end{theorem}

\begin{proof}
Suppose that for some $w$ in $V$, $\M_w^L=0$ or $\M_w^R=0$ . Then it follows from Lemma \ref{M_u_dimension_bound} that $\dim \M\leq n$, and there is nothing more to prove.

We may therefore assume that $\M_w^L\neq 0$ and $\M_w^R\neq 0$ for all $w\in V$.
The required dimension bound certainly follows from Theorem \ref{alternating_constant_rank_bound} if 
$\M$ is a subspace of $\Alt(V)$ and thus we may additionally assume that $\M$ is not contained in
$\Alt(V)$. It follows then from Theorem
\ref{bilinear_dimension_bound} that there exists  $u$ in $V$ such that 
$\M_u^L\cap \M_u^R$ contains no elements of rank $m$ and $\dim \M\leq 2m-1 +\dim (\M_u^L\cap \M_u^R)$. In view of the constant rank $m$ hypothesis, this implies that $\M_u^L\cap \M_u^R=0$ and hence
\[
\dim \M\leq 2m-1,
\]
as asserted.
\end{proof}

Following the strategy of previous sections, we may now prove an upper bound for the dimension of a subspace $\M$ of $\Bil(V)$ in terms of $|\rank(\M)|$.

\begin{theorem} \label{bilinear_several_ranks_bound}

Let $\M$ be a non-zero 
subspace of $\Bil(V)$, let $m$ be the largest integer in $\rank(\M)$ and let
$r=|\rank(\M)|$. Then if
$m\leq \lceil n/2\rceil$
and $|K|\geq m+1$, we have
\[
\dim \M\leq rn.
\]
\end{theorem}

\begin{proof}
We proceed by induction on $r$.
If $\M\leq \Alt(V)$, we already have the inequality
\[
\dim \M\leq rn-\frac{r(r+1)}{2}
\]
from Theorem \ref{alternating_several_ranks_bound} and our proposed inequality is clearly a consequence of this sharper inequality.  We may therefore assume that $\M$ is not contained in $\Alt(V)$. 

Suppose next that for some $w$ in $V$, $\M_w^L$ contains no element of rank $m$.
Then we certainly have $\dim \M_w^L\leq (r-1)n$ by induction and Lemma \ref{M_u_dimension_bound} yields $\dim \M\leq rn$, as required. An identical argument
applies if $\M_w^R$ contains no element of rank $m$.

We may thus assume that for all $w\in V$, both $\M_w^L$ and $\M_w^R$ contain elements
of rank $m$. Theorem \ref{bilinear_dimension_bound} shows that in this case there exists an element $u$ in $V$ such that 
$\M_u^L\cap \M_u^R$ contains no elements of rank $m$ and the inequality 
\[
\dim \M\leq 2m-1+\dim (\M_u^L\cap \M_u^R)
\]
holds. 

We certainly have $|\rank(\M_u^L\cap \M_u^R)|\leq r-1$ and we may identify
$\M_u^L\cap \M_u^R$ with a subspace of $\Bil(W)$, where $W$ is a subspace of $V$ of dimension $n-1$. Now if $m'$ is the largest integer in $\rank(\M_u^L\cap \M_u^R)$, we have $m'\leq m-1$ and it is easy to verify that, since 
$m\leq \lceil n/2\rceil$, the inequality $m'\leq \lceil (n-1)/2\rceil$ also holds.

It follows by induction that 
\[
\dim (\M_u^L\cap \M_u^R)\leq (r-1)(n-1).
\]
Furthermore, since $m\leq \lceil n/2\rceil$, we have $2m-1\leq n$. Thus, our two displayed inequalities above lead to
\[
\dim \M\leq n+(r-1)(n-1)=rn-(r-1)<rn.
\]
This inequality completes the induction step and establishes the theorem.
\end{proof}

Theorem \ref{bilinear_several_ranks_bound} is optimal in non-trivial cases, as we shall now explain. Suppose that $K$ has a separable extension $L$, say, of finite degree $s>1$. Let $W$ be a vector space of dimension $m$ over $L$. Given any integer $r$ satisfying $1\leq r\leq m$, $\Bil(W)$ contains a subspace $\N$, say, of  dimension
$rm$ in which all elements have rank at most $r$. We may realize $\N$ by taking all $m\times m$ matrices with entries in $L$ whose first $r$ columns are arbitrary and whose last $m-r$ columns are all zero.

Let $V$ denote $W$ considered as a vector space of dimension $n=ms$ over $K$. Given $f$ in $\N$, we define $F$ in $\Bil(V)$ by setting $F(u,v)=\Tr(f(u,v))$ for all $u$ and $v$ in $V$, where $\Tr:L\to K$ is the trace form. We may verify that, since $\Tr$ is non-zero under the separability hypothesis, $\rank F=s\rank f$. The set of all elements $F$, as $f$ runs over $\N$, is then a subspace $\M$, say, of $\Bil(V)$ 
of dimension $rms=rn$, in which $|\rank(\M)|=r$. The elements of $\M^\times$ have rank $ts$, where $1\leq t\leq r$. This example shows that the dimension bound
in Theorem \ref{bilinear_several_ranks_bound} is optimal in this case.

It is possible to provide more detail about the structure of a subspace whose dimension
equals the upper bound given in Theorem \ref{bilinear_several_ranks_bound}, as we shall now
demonstrate.

\begin{theorem} \label{subspace_meeting_upper_bound}
Let $\M$ be a non-zero 
subspace of $\Bil(V)$, let $m$ be the largest integer in $\rank(\M)$ and let
$r=|\rank(\M)|$. Suppose that $\dim \M=rn$,
$m\leq \lceil n/2\rceil$
and $|K|\geq m+1$. Then for each integer $s$ satisfying $1\leq s\leq r$, there exists a subspace $\M_s$, say, of $\M$ such that $\dim \M_s=sn$ and $\rank(\M_s)$ consists of the
$s$ smallest integers in $\rank(\M)$.
\end{theorem}

\begin{proof}
We first note that $\M$ is not contained in $\Alt(V)$, since otherwise we have
$\dim \M<rn$ by Theorem \ref{alternating_several_ranks_bound}. 
We claim that there is some $w\in V$ such that either $\M_w^L$ or $\M_w^R$ contains no element of rank $m$. For if this is not the case, Theorem \ref{bilinear_dimension_bound} shows that we can find $u\neq 0$ such that $\M_u^L\cap \M_u^R$ contains no element
of rank $m$ and
\[
\dim(\M_u^L\cap \M_u^R)\geq \dim \M-(2m-1).
\]

Now $|\rank(\M_u^L\cap \M_u^R)|\leq r-1$, since $\M_u^L\cap \M_u^R$ contains no element
of rank $m$, and hence we have $\dim (\M_u^L\cap \M_u^R)\leq (r-1)n$ by Theorem \ref{bilinear_several_ranks_bound}. It follows that 
\[
(r-1)n\geq rn-(2m-1)
\]
and hence $n\leq 2m-1$. This inequality is incompatible with the hypothesis that $m\leq \lceil n/2\rceil$ and hence establishes our claim above.

We may as well assume therefore that there is some $w$ such that $\M_w^L$ contains no
element of rank $m$ and thus $|\rank(\M_w^L)|\leq r-1$. Lemma \ref{M_u_dimension_bound}
implies that $\dim \M_w^L\geq (r-1)n$, whereas Theorem \ref{bilinear_several_ranks_bound}
implies that $\dim \M_w^L\leq (r-1)n$. It follows that $\dim \M_w^L=(r-1)n$
and $|\rank(\M_w^L)|= r-1$. We now complete the proof by reverse induction on $r$.
\end{proof}

It would be valuable to know if Theorem \ref{bilinear_several_ranks_bound}
holds in cases where the hypothesis $m\leq \lceil n/2\rceil$ is weakened. We describe below a simple example of a positive answer.

\begin{theorem} \label{two_ranks_bound}

Suppose that $n$ is even and $\M$ is a subspace of $\Bil(V)$ such that 
$\rank \M=\{ n/2, n\}$. Then if $|K|\geq n/2+1$, we have $\dim \M\leq 2n$.

\end{theorem}

\begin{proof}
Let $u$ be any  element of $V^\times$. $\M_u^L$ cannot contain any elements of rank $n$ and hence must be a constant rank $n/2$ subspace of $\Bil(V)$. Theorem \ref{bilinear_constant_rank_bound} implies that $\dim \M_u^L\leq n$ and then
Lemma \ref{M_u_dimension_bound} implies that $\dim \M\leq 2n$, as desired.
\end{proof}

The discussion before this theorem shows that the bound of Theorem \ref{two_ranks_bound} can be optimal in non-trivial ways.

\end{document}